\documentclass{amsart}
\usepackage{graphicx}
\usepackage{graphicx,esint}
\usepackage{amssymb,amscd,amsthm,amsxtra}
\usepackage{latexsym}
\usepackage{epsfig}

\newtheorem{thm}{Theorem}[section]
\newtheorem{cor}[thm]{Corollary}
\newtheorem{lem}[thm]{Lemma}

\theoremstyle{definition}
\newtheorem{defn}[thm]{Definition}
\theoremstyle{remark}

\numberwithin{equation}{section}

\newcommand{\R}{\mathbb R}

\newcommand{\eps}{\epsilon}

\newcommand{\p}{\partial}

\newcommand{\comment}[1]{}

\begin{document}

\title{A gradient bound for free boundary graphs}%
\author{D. De Silva}
\address{Department of Mathematics, Barnard College, Columbia University, New York, NY 10027}
\email{\tt  desilva@math.columbia.edu}
\author{D. Jerison}%
\address{Department of Mathematics, Massachusetts Institute of Technology, Cambridge, MA }%
\email{\tt  jerison@math.mit.edu}%

\thanks{Both authors were partially supported by NSF grant DMS-0244991.}

\begin{abstract} We prove an analogue for a one-phase free boundary problem
of the classical gradient bound for solutions to the minimal surface
equation.  It follows, in particular, that every energy-minimizing free 
boundary that is a graph is also smooth.  The method we use also leads to
a new proof of the classical mimimal surface gradient bound.
\end{abstract}
\maketitle

\section{Introduction}

Let $\Omega$ be a domain in $\R^n$, and consider
the one-phase free boundary problem ($u\ge 0$)
\begin{equation}\label{fbintro} \left \{
\begin{array}{ll}
    \Delta u =0,   & \hbox{in $\Omega^+(u):= \{x \in \Omega : u(x)>0\}$,} \\
    |\nabla u| =1, & \hbox{on $F(u):= \partial \Omega \cap \Omega^+(u)$.} \\
\end{array}\right.
\end{equation}
The set $F(u)$ is known as the free boundary.  There are strong parallels
between the theory of these hypersurfaces and the theory of minimal
surfaces.  The existence of solutions $u$ and partial regularity
(smoothness almost everywhere with respect to surface measure)
of the free boundary $F(u)$ was proved by
Alt and Caffarelli \cite{AC}.

We will begin by formulating our main result in the special case of 
energy-minimizing solutions.  We call $u$ energy-minimizing on $\bar \Omega$
if $u$ minimizes the functional
\[
J(v)=
\int_{\Omega}(|\nabla v|^2 + \chi_{\{v>0\}})dx,
\]
among all functions with the same boundary values as $u$.  The first variation
(Euler-Lagrange) equations for $u$ are \eqref{fbintro}.  Indeed, it
is easy to show that $u$ is harmonic in $\Omega^+(u)$, and it follows from
deeper results of \cite{AC,C1,C2} that $u$ satisfies the free boundary 
condition $|\nabla u| = 1$ on $F(u)$ in a viscosity sense, defined in Section 2.

Define a cylinder of height $2L$, with base the ball of radius
$B_r$ in $\mathbb{R}^{n-1}$, by
\[
\mathcal{C}(r,L) := B_r \times (-L,L) \subset \mathbb{R}^{n};
\quad
\mbox{(and }\mathcal{C}_L := \mathcal{C}(1,L) ).
\]

\begin{thm}\label{intromain} If an energy-minimizing solution $u$ on the
cylinder $\mathcal{C}_L$ is monotone in the vertical direction,
\[
\partial u/\partial x_n \ge 0 \ \mbox{ on }\  \mathcal{C}_L^+(u),
\]
and its free boundary $F(u)$ is a fixed distance from
the top and bottom of the cylinder, i.~e.,
\[
F(u) \subset \mathcal{C}_{L-\epsilon} \quad \mbox{ for some } \epsilon >0,
\]
then $F(u)$ is the graph of a smooth function $\varphi$,
\[
F(u) = \{(x,y): x\in B_1; \quad y = \varphi(x)\}
\]
with
\[
\sup_{x\in B_{1/2}} |\nabla \varphi(x)| \le C
\]
for a constant $C$ depending only on $L$, $\epsilon$, and $n$.
\end{thm}

Let us compare our theorem  with the classical gradient bound
on minimal surfaces due to Bombieri, De Giorgi, and Miranda, 
which can be stated as follows.
\begin{thm}\label{boundMS} \cite{BDM}  Let $\phi \in C^\infty(B_1)$ be a
solution to the minimal surface equation \begin{equation}\label{MS}
\textrm{div}\left(\frac{\nabla \phi}{\sqrt{1+|\nabla \phi|^2}}\right)= 0 \quad
\text{in} \quad B_1, \end{equation}with $|\phi|\leq M.$ Then
\begin{equation}
|\nabla \phi| \leq C \quad \text{in} \quad {B}_{1/2}\end{equation}
with $C$ depending on $n$ and $M$. \end{thm}

\noindent
The hypothesis $\partial u/\partial x_n \ge 0$ in Theorem \ref{intromain}
implies, by the strong maximum
priniciple, that $\partial u/\partial x_n > 0$.  Therefore the level surfaces
$\{x: u(x) =c\}$ for $c>0$ are graphs.  The hypothesis that the free boundary
is a fixed distance from the top and bottom of the cylinder replaces the hypothesis
in Theorem \ref{boundMS} that the oscillation of the function $\phi$ is bounded by 
$M$.   Furthermore, the minimal surface equation \eqref{MS} implies
that the graph of $\phi$ is area-minimizing, so that the assumption in 
Theorem \ref{intromain} that the free boundary is energy minimizing is
analogous.

In the theory of
minimal surfaces, it is well-known that minimal graphs are real analytic
in the interior of the their domain of definition.
The key first step in the proof of full regularity of the minimal
graphs is to establish that the graph is Lipschitz, that is, the graph
of a function with a bounded gradient.   The gradient bound proved here leads, likewise, 
to full regularity.  If the free boundary is a Lipschitz graph, then
Caffarelli \cite{C1} proved that the graph is $C^{1,\alpha}$ for some
$\alpha>0$. Higher regularity results of \cite{KN} then yield the local
analyticity of $F(u)$.  So real analyticity follows if one can confirm
the Lipschitz property, i.~e., the gradient bound.


In \cite{D2}, an a priori gradient bound for
smooth free boundary graphs is proved in the case when
$n=2,3.$  The proof given there is also motivated by the strong analogy
with minimal surfaces, but is completely different.  An advantage of the
results here is that because they work in all dimensions, they can
be expected to apply to the free boundary analogue of the Bernstein problem.
The application
we have in mind is to the construction (as yet unrealized) of a global solution
to the free boundary problem (other than the obvious solution $u(x) = x_1^+$)
whose level surfaces are graphs.  This would
be analogous to the counterexample to the Bernstein conjecture ---
a complete non-planar mimimal graph constructed in \cite{BDG} in $\mathbb{R}^9$.
In \cite{DJ}, it is shown that a certain cone in $\R^7$ is the free boundary
analogue of the Simons cone in minimal surface theory.  Based on this example,
one should expect to find a free boundary whose level surfaces are
non-flat graphs in $\mathbb{R}^8$.

The theorem whose proof occupies most of this paper has a more technical statement.
See Section 2 for the definition
of a viscosity solution and nontangentially accessible (NTA) domains.
\begin{thm}\label{main}  Let $u$ be a viscosity solution to \eqref{fbintro}
in the cylinder $\mathcal{C}_L$.  Suppose that $u$  is monotone in the
vertical direction,
\[
\partial u/\partial x_n \ge 0 \ \mbox{ on }\  \mathcal{C}_L^+(u),
\]
and its free boundary is given as the graph of a continuous function
$\varphi$, $F(u)=\{(x,y): x\in B_1; \quad y = \varphi(x)\}$.  Suppose
that the oscillation of $\varphi$ is bounded,
\[
\max_{x\in B_1}|\varphi(x)| \le L-1,
\]
and, finally, that there is a nontangentially accessible (NTA)
domain $\mathcal D$ such that $$
\mathcal{C}\left(\frac{9}{10},L-\frac{1}{2}\right) \cap \mathcal{C}^+_L(u) \subset \mathcal D \subset \mathcal{C}^+_L(u).$$
Then
\[
\sup_{x\in B_{1/2}} |\nabla \varphi(x)| \le C
\]
for a constant $C$ depending only on $L$, the NTA constants, and $n$.
\end{thm}


Theorem \ref{intromain} will follow from Theorem \ref{main} using results
of  \cite{D1}.  Roughly 
speaking, \cite{D1} shows that the hypotheses of Theorem \ref{intromain}
imply the hypotheses of Theorem \ref{main}.  In particular, a key
estimate from \cite{D1} is that the positive phase satisfies an NTA
property on any smaller cylinder.  Moreover, it is also proved in \cite{D1} that 
under the hypotheses of Theorem \ref{intromain},
the free boundary is the graph of a continuous function $\varphi$.

The proof of Theorem \ref{main} is based on comparing $u(x)$ to
its vertical translates $u(x + te_n)$.  One constructs a family of
supersolutions related to $u(x + te_n)$ and uses a deformation maximum
principle argument to show that $u(x+te_n) \ge u(x) + ct$ for
sufficiently small $t>0$.  The function $u(x)$ is comparable to the
distance from $x$ to the free boundary.  The estimate shows that the
change in $u$ in the vertical direction is comparable to the change in
$u$ in the direction normal to each level surface, which is equivalent
to a Lipschitz bound on the graph of the level surface.

The construction of the family of supersolutions makes use of the
basic estimates on NTA domains which were the reason the notion of NTA
was introduced in \cite{JK}.  The NTA property
guarantees that every positive harmonic function that vanishes on the
boundary vanishes at the same rate as $u$.  The NTA property was first
used in connection with regularity of free boundaries by Aguilera,
Caffarelli and Spruck \cite{ACS}, who
proved a partial regularity result.
The NTA property also holds for the singular conic solution of Alt and
Caffarelli.  (This cone is not a graph, of course.  Otherwise it would
contradict Theorem \ref{main}.)

Our proof of the gradient bound for free boundaries leads to a new
proof of the classical gradient bound for minimal graphs.
This new proof of Theorem \ref{boundMS} is related to a much simpler
proof due to N. Korevaar \cite{Ko}.
The hope is that this new method, while more complicated than the method
in \cite{Ko},  will ultimately apply to
classes of semilinear problems that include both free boundary problems
and minimal surface problems as singular limits.   An interesting
aspect of our proof is that it deepens the analogy
between minimal surfaces and free boundaries. 

The paper is organized as follows. In Section 2 after briefly
recalling some standard definitions and known results, we prove
Theorem \ref{main} and deduce Theorem \ref{intromain}.
We present our proof of Theorem \ref{boundMS} in Section 3. In Section 4,
we examine the parallels between the two proofs and especially
between two key parallel ingredients, namely 
the boundary Harnack inequality for NTA domains and 
the intrinsic Harnack inequality of Bombieri
and Giusti \cite{BG}.

\section{Gradient bound for free boundary graphs}

\subsection{Preliminaries.}
We recall the definition of a viscosity solution \cite{C1}.
\begin{defn}\label{defsol} Let $u$ be a nonnegative continuous function in
$\Omega$. We say that $u$ is a viscosity solution to \eqref{fbintro} in
$\Omega$ if and only if the following conditions are
satisfied:
\begin{enumerate}
\item $\Delta u = 0$ in $\Omega^+(u)$;
\item If $x_0 \in F(u)$ and $F(u)$ has at $x_0$ a
tangent ball $\mathcal{B}_{\epsilon}$ from either the positive or the zero
side, then, for $\nu $ the unit radial direction of $\partial
\mathcal{B}_{\epsilon}$ at $x_0$ into $\Omega^+(u)$,
\begin{equation}\nonumber
u(x)= \langle x-x_0,\nu\rangle^+ + o(|x-x_0|), \ \text{as $x\rightarrow x_0.$}
\end{equation}
\end{enumerate}
\end{defn}
Standard elliptic
regularity theory implies that if $F(u)$ is a smooth surface
near $x_0$, then $u$ is smooth up to
the free boundary near $x_0$ and the free boundary condition $|\nabla u|=1$ is
valid in the classical sense in such a neighborhood.

Denote by $d(x)=\text{dist}(x,F(u))$.  In this section, the balls
$\mathcal{B}_r = \mathcal{B}_r(0)$ and $\mathcal{B}_r(x)$ will be
in $\R^n$ while the balls $B_r(x)$ will be in $\R^{n-1}$.  The
following result follows easily from the Hopf lemma and interior
regularity of elliptic equations (see for example
\cite{CS},\cite{D2}).
\begin{lem}\label{Lipconst}Let $u$ be a viscosity solution to
\eqref{fbintro} in $\mathcal{B}_1$, $0 \in F(u)$. Then,
$u$ is Lipschitz continuous in $\mathcal{B}_{1/2}$ and there is a dimensional
constant $K$ such that
\[
\sup_{\mathcal{B}_{1/2}}|\nabla u| \le K,
\]
and
\[
u(x) \leq K d(x), \ \ \ \text{for all $x \in \mathcal{B}_{1/2}$}.
\]
\end{lem}








\begin{defn}\label{nondeg} We say that a viscosity solution $u$
is nondegenerate in $\mathcal{B}_1$ if there is a constant $c>0$ such that
$u(x)\ge cd(x)$ for all $x\in \mathcal{B}_1^+(u)$.
\end{defn}

\smallskip

We now recall the notion of nontangentially accessible (NTA)
domains.

\begin{defn}\label{NTA}A bounded domain $D$ in $\R^n$ is called NTA, when there exist
constants $M$ and $r_0 >0$ such that:
\begin{enumerate}
\item Corkscrew condition. For any $x \in \p D,$ $r < r_0,$ there exists $y=y_r(x) \in D$ such
that $M^{-1} r < |y-x|<r$ and $\text{dist}(y,\p D) > M^{-1}r;$
\item The Lebesgue density of $D^{c}$ at any of its points is bounded below uniformly by a positive constant $c$, i.e for
all $x \in \partial D, 0 < r <r_0,$ $$\frac{|\mathcal{B}_r(x) \setminus D
|}{|\mathcal{B}_r(x)|} \geq c;$$
\item Harnack chain condition.
If $\eps >0$ and $x_1, x_2$ belong to $D$, $\text{dist}(x_j,\p
D)>\eps$ and $|x_1- x_2|< C_1\eps,$ then there exists a sequence of
$C_2$ balls of radius $c\epsilon$ such that the first ball is centered
at $x_1$, the last at $x_2$, such that the centers of consecutive balls
are at most $c\epsilon/2$ apart.  The number of balls $C_2$ in the chain
depends on $C_1$, but not on $\epsilon$.
\end{enumerate}
\end{defn}

We recall some results about NTA domains \cite{JK}. We start with
the following boundary Harnack principle for harmonic functions.

\begin{thm}\label{BHP}(Boundary Harnack principle) Let $D$ be an $\textrm{NTA}$ domain
and let $V$ be an open set. For any compact set $K \subset V,$
there exists a constant $C$ such that for all positive harmonic
functions $u$ and $v$ in $D$ vanishing continuously on $\partial D
\cap V,$ and $x_0 \in D \cap K,$ $$C^{-1}\frac{v(x_0)}{u(x_0)}u(x)
\leq  v(x) \leq C\frac{v(x_0)}{u(x_0)}u(x), \ \  \textit{for all}
\ \ x \in K \cap \overline{D}.
$$
\end{thm}

\smallskip
The boundary Harnack inequality above will be our main tool in the
proof of Theorem \ref{main}. We will also need some further facts.
First, recall that for any bounded domain $D \subset \R^n$ and any
arbitrary $y_0 \in D$, one can define the harmonic measure
$\omega^{y_0}$ of $D$ evaluated at $y_0$ (for the definition see
for example \cite{JK}). We note that for any $y_1,y_2 \in D$, the
measures $\omega^{y_1}$ and $\omega^{y_2}$ are mutually absolutely
continuous. Hence, from now on we fix a point $y_0 \in D$ and
denote $\omega=\omega^{y_0}.$

A nontangential region at
$x_0 \in \partial D$ is defined as
$$\Gamma_\alpha(x_0)=\{x \in D : |x-x_0|<(1+\alpha)\text{dist}(x,\partial
D)\}.$$ Let $u$ be defined on $D$ and $f$ on $\partial D$. We say
that $u$ converges to $f$ nontangentially at $x_0 \in
\partial D$ if for any $\alpha,$
$$\lim_{x\rightarrow x_0} u(x)=f(x_0) \quad \text{for} \ x\in
\Gamma_\alpha(x_0).$$

The following Fatou-type theorem was proved in \cite{JK}.
\begin{thm}\label{Fatou} Let $D$ be an NTA domain. If $u$ is a positive harmonic
function in $D$, then $u$ has finite nontangential limits for
$\omega$-almost every $x_0 \in \partial D.$
\end{thm}

We deduce from this the following regularity result for NTA free
boundaries.
\begin{lem}\label{full harmonic measure}Let $u$ be a viscosity
solution to \eqref{fbintro} in $\mathcal{B}_1$, $u$  non-degenerate in $\mathcal{B}_{3/4}$, and
$0 \in F(u)$. Assume that there is an NTA domain $D$ such that
$D\subset \mathcal{B}_1^+(u)$ and $F(u)\cap \mathcal{B}_{3/4}\subset \partial D$.
Then, $F(u)\cap \mathcal{B}_{1/2}$ is smooth almost everywhere with
respect to harmonic measure $\omega$ of $D$.
\end{lem}

\begin{proof} Since each partial derivative  $\partial u/\partial x_j$
is a bounded harmonic
function, Theorem \ref{Fatou} implies that for $\omega$-almost every
$x_0\in F(u)\cap \mathcal{B}_{1/2}$, there exists $a\in \R^n$ such that for
every $\alpha<\infty$, $\nabla u(x)\to a$ as $x\to x_0$, for $x\in  \mathcal{B}_1^+(u)$,
$|x-x_0| <
(1+\alpha)\mbox{dist}(x,F(u))$.   We will prove that
$F(u)$ is flat and hence smooth in a neighborhood of $x_0$.
The idea of the proof is to show that for $x$ near $x_0$, $u$ is close
to a linear function with gradient $a$.  Provided that $a$ is not the zero
vector, this will show us that the level sets of $u$ are flat and hence (by
\cite{AC, C2}) that the free boundary is smooth near $x_0$.

For notational simplicity assume $x_0=0$.  Denote
by $u_r$ the rescaling of $u$, $u_r(x)=u(rx)/r$.  We will use the notation
$A_1 \approx A_2$ for positive numbers that are comparable modulo constants
that depend only on the NTA constants and the ratio of $u(x)$
to the distance to the free boundary (bounded above and below by
Lemma \ref{Lipconst}
and nondegeneracy).  Consider a point $z\in \mathcal{B}_1^+(u_r)$
such that $u_r(z) \approx 1$.  Note that although the point $z$ depends
on $r$, we require the constants in comparability of $u_r(z)$ with $1$ to be
independent
of $r$ as $r\to0$.   For any $x\in \mathcal{B}_1\cap \overline{\{u_r > 0\}}$, the
NTA properties imply there is a (nontangential, corkscrew) path $p(t)$
such that $p(0)=x$, $p(1)= z$, $|p'(t)|\le C$ and
\[
u_r(p(t)) \approx \mbox{dist}(p(t),F(u_r)) \approx  t + u_r(p(0))
\]
independent of $r$.

Fix $C_2 <<C_1<\infty$, $\delta>0$, and  denote,
\[
T_\delta^r =\{x \in \mathcal{B}_{C_1}^+(u_r) : u_r(x) \ge \delta\},
\]
\[
\Gamma_{\alpha}^r=\{x \in \mathcal{B}_{1/r}^+(u_r): |x|<
(1+\alpha)\text{dist}(x,F(u_r))\}.
\]
Since $u_r(x)$ is comparable to the distance from $x$ to $F(u_r)$,
for any $x\in T_\delta^r\cap \mathcal{B}_{C_2}$, there is a constant $c>0$ such that
the path $p(t)$ from $x$ to $z$ belongs to $T_{c\delta}^r$.
Choose $\alpha$ sufficiently large depending on $\delta$ and $c>0$
and $r$ sufficiently small depending on $C_1$ such that
\[
T_{c\delta}^r \subset \Gamma_{\alpha}^r.
\]
Thus there is $r_0>0$ (depending on $C_1$, $\delta$,
and $\alpha$) such that that for $r<r_0$,
\[
|\nabla u_r(x) - a| < \delta, \quad \mbox{for } x\in T_{c\delta}^r.
\]
Define a linear function of $x$, by $L(x) = u_r(z) + a\cdot(x-z)$.
For all $x\in T_\delta^r\cap \mathcal{B}_{C_2}$, since $u_r(z) - L(z) = 0$, and
$p(t)\in T_{c\delta}^r$ ,
\[
|u_r(x) - L(x)| =
\left|\int_0^1 (\nabla u_r(p(t))-a)\cdot p'(t)dt\right| \le C_3\delta.
\]
In all, we have shown that for every $x\in \mathcal{B}_{C_2}$ such that $u_r(x)\ge \delta$,
\[
|u_r(x) - L(x)| \le C_3\delta.
\]

Next, we deduce that $|a|\approx 1$.  (The upper bound $|a|\le K$
already follows from the upper bound on $|\nabla u|$.) Since
$u_r(0)=0$ and $u_r(z)\approx 1$, for some $0<t <1$, the point
$x=tz$ satisfies $u_r(x)=\delta$.  So $x\in T_\delta^r\cap
\mathcal{B}_{C_2}$ and $|\delta - u_r(z) -a\cdot(x-z)|\le
C_3\delta$. Hence, $|a| \ge |a\cdot(x-z)| \ge u_r(z) - \delta -
C_3\delta \ge u_r(z)/2$. (All we need in what follows is that $a$ is
bounded and nonzero.)

We can now conclude that the free boundary is flat in the appropriate
sense.  Consider a point $x\in F(u_r)\cap \mathcal{B}_1$ and its path $p(t)$ to $z$.
There is $t>0$ such that $u_r(p(t))=\delta$.  Denote $y=p(t)$. Then
$|y-x|\le C\delta$ and
$y\in T_\delta\cap \mathcal{B}_{C_2}$.  The preceding
argument says $|u_r(y)-L(y)| \le C_3\delta$.  Therefore,
\[
|L(x)| \leq |L(y)| + |L(x)-L(y)| \le |u_r(y) -L(y)| + |u_r(y)| + |a\cdot(x-y)| \le C_4\delta
\]
for a larger constant $C_4$.   Since $a$ is bounded
away from $0$ in length, the bound on $L(x)$ implies that every point
of $F(u_r)\cap \mathcal{B}_1$
is within a distance a constant times $\delta $ of the plane $L(x)=0$.
For sufficiently small $\delta$, this flatness condition implies smoothness
of the free boundary (see \cite{AC,C2}).

\end{proof}

\subsection{The proof of Theorem \ref{main}.}
Throughout the proof, $c_i, C_i$ denote constants depending
on $L,n$, and possibly on the NTA constants.  Also, a point $x \in \mathbb{R}^n$ may be denoted by
$(x',x_n),$ with $x'=(x_1,\ldots,x_{n-1})$.

We divide the proof in three steps.

\medskip

\noindent \textbf{Step 1: Nondegeneracy and separation of level sets.}

\medskip

We show first the nondegeneracy of $u$, namely that
if $\mathcal{B}_\rho(x_0) \subset C_L^+(u)$, $\rho <1$, then
\begin{equation}\label{nondegeq}
u(x_0) \ge \gamma_n\rho
\end{equation}
for a dimensional constant $\gamma_n>0$.

Denote by
$g$ a strictly superharmonic function on the annulus $E = \mathcal{B}_2\backslash \mathcal{B}_1$\
such that
\[
  \begin{cases}
    g= a_n & \text{on $\partial \mathcal{B}_2$}, \\
    g= 0 & \text{on $
 \partial \mathcal{B}_1 $},\\
    |\nabla g| < 1 & \text{on $\partial \mathcal{B}_1$,}
  \end{cases}
\] with $a_n>0$ small dimensional constant.
Let $r = \rho/4$. Denote $g_r(x) = rg(x/r)$, and
\[
h_t(x) = g_r(x-x_0-te_n)
\]
defined on the closed annulus $E_t = \bar{\mathcal{B}}_{2r}(x_0 +
te_n) \backslash \mathcal{B}_r(x_0+te_n)$. For $t$ sufficiently
small, $E_t \subset \{x: -L < x_n < -L +1\}$ so that $h_t(x) \ge 0
= u(x) $ for $x\in E_t$.  Increasing $t$ translates the region
$E_t$ upwards.  Let $t_0$ be the least $t$ for which the graph of
$h_t$ touches the graph of $u$, i.~e., so that there is a point
$z_0\in E_t$ for which $h_{t}(z_0) = u(z_0) > 0$.  Because $h_t$
is a strict supersolution the point $z_0$ belongs to the outer
boundary, $z_0\in \partial \mathcal{B}_{2r}(x_0 + t_0e_n)$.
Furthermore, because the free boundary of $u$ and $h_t$ can't
touch, $t_0 \le  -\rho - r < 0$.  Monotonicity of $u$ implies
$u(z_0 - t_0e_n) \ge u(z_0) = h_{t_0}(z_0) = a_nr$.  Finally,
since $|z_0-t_0e_n -x_0| = 2r = \rho/2$,  Harnack's inequality
comparing the value of $u$ at $z_0 - t_0 e_n$ and $x_0$ implies
that there is a dimensional constant $\gamma_n>0$ such that
$u(x_0) \ge \gamma_n\rho$, as required.

Next, we will show that level sets near the top of the cylinder are
separated by an appropriate amount.  Let $\epsilon >0$ and denote by
\[
v(x)=u(x-\epsilon e_n).
\]
Since $u$ is strictly
monotone in the vertical direction, $v(x) < u(x)$ on
$\mathcal{C}^+_{L}(u)$.  We claim that
\begin{equation}\label{v<u}
v(x) \leq u(x) - c_1\epsilon
\ \ \ \textrm{on $B_{9/10}(0)
\times \{L-1/2\}$}
\end{equation}
for $\epsilon < \epsilon_n$ a dimensional constant, and a constant
$c_1>0$ depending only on $L$ and $n$.  To prove \eqref{v<u},
note first that from \eqref{nondegeq} it follows that $u(x) \ge b_n$
for all $x\in
B_{9/10}(0) \times \{L-1/2\}$.  Write $x = (x',L-1/2)$ and let
$t_n$ be such that $u(x',t_n) = b_n/2$, then by monotonicity
$u(x',t) \ge b_n/2$ for all $t \ge t_n$. Consider the segment from
$(x',t_n)$ to $(x',L-1/2)$. It follows from the Lipschitz bound
(Lemma \ref{Lipconst}) that the distance from
any point of the segment
to the free boundary is greater than a dimensional constant.  Thus
by Harnack's inequality the values of $w(x) = (\partial/\partial
x_n)u(x)$ on this segment are comparable with a constant depending
only on $n$ and $L$.  Furthermore,
\[
b_n - b_n/2 \le u(x)-u(x',t_n) = \int_{t_n}^{L-1/2} w(x',t) dt.
\]
Therefore, the minimum of $w$ on this segment is bounded below by a constant
$c_1>0$, depending only on $n$ and $L$.  In particular,
\[
u(x)-v(x) = \int_{L-1/2 - \epsilon}^{L-1/2} w(x',t) dt \ge c_1 \epsilon.
\]

\medskip

\noindent \textbf{Step 2: Construction of a family of
supersolutions.}

\medskip

The hypothesis of Theorem \ref{main} implies (by the construction of
P. W. Jones \cite{J}) that there is an NTA domain between any
pair $\mathcal{C}(r_1,L-a_1)$ and $\mathcal{C}(r_2,L-a_2)$ for $r_1 < r_2 \leq 9/10$
and $a_1 > a_2 \geq 1/2$.  Thus the boundary Harnack inequality,
Theorem \ref{BHP}, has the following corollary.
\begin{cor}\label{corharnack} Let $u$ be as in Theorem \ref{main} and
let $r_1 < r_2 \le 9/10$ and $a_1 > a_2 \ge 1/2$. Then
there is a constant $A$ depending on $L$, the NTA constants of $\mathcal D$,
$r_2-r_1>0$, and $a_1-a_2>0$ such that if
$h_1$ and $h_2$ are positive harmonic
functions on $\mathcal{C}(r_2,L-a_2)\cap \mathcal{C}_L^+(u)$, vanishing on $\p D \cap \mathcal{C}(r_2,L-a_2)$ then
\[
h_1(x)/h_2(x) \le Ah_1(y)/h_2(y)
\]
for every $x$ and $y$ in $\mathcal{C}(r_1,L-a_1)\cap \mathcal{C}_L^+(u)$.
\end{cor}

In this step we start our analysis on the cylinder
$\mathcal{C}(9/10,L-1/2)$ which by abuse of notation we denote by
$\mathcal{C}_1$. Then we restrict to smaller cylinders
$\mathcal{C}_2, \mathcal{C}_3$ with base $B_{8/10}$ and $B_{7/10}$
respectively, height $M$ with $L-1< M < L-1/2$ and $\mathcal{C}_3
\subset \subset \mathcal{C}_2 \subset \subset \mathcal{C}_1$.

Let $w$ be the harmonic function in $\mathcal{C}_1^+(u),$
satisfying the following boundary conditions:
\begin{align}
&w=0, \ \ \   \textrm{on $F(u)$},\\
&v < w \leq u, \ \ \    \textrm{on $
\overline{\mathcal{C}^+_{1}(u)} \cap
\partial \mathcal{C}_{1}$,}\\\label{v+<w<u-} &v + \frac{c_1}{4}\eps < w <
u - \frac{c_1}{4}\epsilon, \ \ \ \textrm{on
$B_{9/10} \times \{L-1/2\}$}.
\end{align}

Notice that
\eqref{v+<w<u-} can be achieved because of the gap \eqref{v<u} between $u$ and $v$.
Since $v$ is subharmonic and  $u$ is harmonic in
$\mathcal{C}^+_{1}(u)$, the maximum principle implies
\begin{equation}\label{w<u}v < w < u \ \ \ \textrm{in $\mathcal{C}^+_{1}(u).$}\end{equation}
Moreover, $\mathcal{C}^+_{1}(w)=\mathcal{C}^+_{1}(u),$
and $F(w)=F(u)\cap \mathcal{C}_{1}$.

We claim next that in the smaller cylinder $\overline{\mathcal{C}}_2$,
\begin{equation}\label{wislip}|\nabla w|(x) \leq C_1, \quad x \in \overline{\mathcal{C}}_2.\end{equation}
Define $d(x) = {\rm dist}(x,F(u))$.  At points
$x \in \overline{\mathcal{C}}_2 \cap \mathcal{C}_1^+(u)$ such that
$d(x) \ge 1/10$, this follows from standard
elliptic regularity and the fact that $w$ is bounded.
On the other hand, at points that are close to $F(u)$, we
have that $B_{d(x)}(x) \subset
\mathcal{C}^+_{1}(u)$ and from Lemma
\ref{Lipconst}, $$w(x) < u(x) \leq K
 d(x).$$ A standard argument using rescaling implies
the bound \eqref{wislip}.


Now, set $h=u-w$. Then $h$ is a positive (see \eqref{w<u})
harmonic function on $\mathcal{C}^+_{1}(u)$ vanishing continuously
on $F(u)$.
Let $H$ be the harmonic function in the cylinder  $B_{9/10} \times
(L-1, L-1/2)$, with boundary data $c_1/2$ on the top of the
cylinder and vanishing on the remaining part of the boundary.
Then, in view of \eqref{v+<w<u-}, $h \geq \eps H$. Thus, $h(x_1)
\geq c_1\epsilon/4$, at $x_1= (L-1/2-\delta_n)e_n$ for a small
dimensional constant $\delta_n >0$. Moreover, by the Lipschizt
continuity of $u$ we get that $h(x_1) < (u - v)(x_1) \leq K \eps$.
Using non-degeneracy and Lipschizt continuity of $u$ we also have
that $b_n \leq u (x_1) \leq 2LK$. Thus, Corollary \ref{corharnack}
gives
\[
c_2 \eps u \leq h \leq C_2\epsilon u \ \ \textrm{on
$ \overline{\mathcal{C}^+_2(u)}$}.
\]
The upper bound on $h$ implies,
\begin{equation}\label{upper}w(x) \geq (1-C_2\epsilon)u(x)  \ \
\textrm{on $ \overline{\mathcal{C}^+_2(u)}$},\end{equation} while
the lower bound gives
\begin{equation}\label{lower}
w(x) \leq (1-c_2\epsilon)u(x) \ \ \textrm{on $
\overline{\mathcal{C}^+_2(u)}$}.
\end{equation}
In particular, if $F(u)$ is smooth around a point $x_0 \in \mathcal{C}_2$ then
$|\nabla u|(x_0)= 1,$
which combined with \eqref{lower} gives
\begin{equation}\label{nabla}|\nabla w|(x_0) \leq 1 - c_2\epsilon.\end{equation}
According to Lemma
\ref{full harmonic measure} we then have
\begin{equation}\label{nabla2}|\nabla w|\leq 1 - c_2\epsilon \quad \text{$\omega$-almost everywhere on
$F(u) \cap \mathcal{C}_2$}.\end{equation}

Next we use \eqref{nabla2} to show that, by restricting on the smaller cylinder $\mathcal{C}_3$,  we have
\begin{equation}\label{strictinterior}
|\nabla w| \leq 1-c_2\epsilon
+ C_3u\ \ \ \textrm{on $\mathcal{C}^+_{3}(u)$}.
\end{equation}
Let $\tilde{h}$ be the largest harmonic function
$\tilde h \le C_1$ in $\mathcal{C}^+_{2}(u)$
such that
\[
\tilde{h}=1-c_2 \epsilon \quad \text{on} \ F(u)\cap \mathcal{C}_2
\]
with $C_1$ the constant in \eqref{wislip}.   Since
$|\nabla w|$ is subharmonic, it satisfies
\eqref{wislip}-\eqref{nabla2}
we get \begin{equation}\label{nablaw<h}
|\nabla w| \leq
\tilde{h}.
\end{equation}

On the other hand, $\tilde{h} - (1-c_2\epsilon)$ is a positive
harmonic function on $\mathcal{C}^+_{2}(u),$ and it is zero on
$F(u)$. Since by non-degeneracy $u$ is bounded below by a
dimensional constant on the top of $\mathcal{C}_3$, Corollary \ref{corharnack}
gives
\[
\tilde{h}
-(1-c_2\epsilon) \leq C_3 u \quad \text{on
$\mathcal{C}^+_{3}(u).$}
\]
Combining this inequality with \eqref{nablaw<h} we obtain
\eqref{strictinterior}.

We now use \eqref{strictinterior} to construct a family of strict
supersolutions. Define for $t\ge0$,
\[
w_t(x)=w(x)-tg(x), \ \  \ \textrm{$x \in \mathcal{C}_{1}$}
\]
with
\[
g(x) = e^{Ax_n}\phi\left(|x'|\right)
\]
where $A$ is a positive constant to be chosen later,
and $\phi\ge0$ is a smooth bump function such that
\[
\phi(r)=\left\{%
\begin{array}{ll}
    1, & \hbox{if $r < 1/2$} \\
    0, & \hbox{if $r \geq  7/10$.} \\
\end{array}%
\right.
\]
Moreover, we will choose $\phi$ such that $\phi(r)>0$ for $r < 7/10$
\[
\phi''(r) + \frac{n-2}{r}\phi'(r) \geq 0, \ \ \text{if \quad $6/10
\leq r \leq 7/10$}.
\]
Indeed, let $\displaystyle \psi(s) = e^{-2n/s}$ for $s>0$ and
$\psi(s) = 0$ for $s\le 0$.  Then for $0 \le s \le 1$,
\[
\psi''(s) - 2n\psi'(s) =  [(2n + 4n^2)/s^2 - (2n)^2/s]e^{-2n/s} \ge 0.
\]
Because $(n-2)/r \le 2n$ for $r \ge 1/2$, the
function $\phi_1(r) = \psi(7/10 - r)$ satisfies
the differential inequality for $\phi$ above in the range $r \ge 1/2$.
Using a partition of unity, $\phi_1$ can be modified
without changing its values for $r \ge 6/10$, to obtain
a function $\phi$ that is equal to $1$ for $r\le 1/2$.
Finally, using the inequalities for $\phi$,
\[
\Delta g = A^2 e^{Ax_n}\phi\left(|x'|\right) +
e^{Ax_n}\Delta \phi\left(|x'|\right)\geq 0
\]
as long as $A$ is a sufficiently large dimensional constant.

Thus, $w_t$ is superharmonic on $\mathcal{C}_{1}^+(w_t).$
Moreover, condition \eqref{strictinterior} together with
\eqref{upper} imply that,
$$|\nabla w_t| \leq |\nabla w| + t|\nabla g| \leq 1-c_2\epsilon +C_4w + t|\nabla
g|, \quad \text{on} \ \mathcal{C}_{3}^+(u).$$ In particular, on
$F(w_t)\cap \mathcal{C}_{3}$, $t>0$, since $w=tg$ we obtain
\[
|\nabla w_t|\leq 1-c_2\epsilon + C_4tg + t|\nabla g|.
\]
Therefore, for $0<t \leq c_3\epsilon$, with $c_3$ small depending
on $c_2$, $C_4$, and $A$, we deduce that
\begin{equation}\label{smallerthan1}
|\nabla w_t|\leq 1-\frac{c_2}{2}\epsilon \quad
\text{on} \  F(w_t)\cap \mathcal{C}_{3}.
\end{equation}

\medskip

\noindent \textbf{Step 3: Comparison.}

\medskip

Observe that because $g$ vanishes on the ``sides''
we have that
\begin{equation}\label{bc1}
w_t = w
> v  \quad \text{on}  \ (\partial B_{9/10} \times
[-L,L]) \cap \overline{\mathcal{C}_{1}^+(w_t)},
\end{equation}
and according to \eqref{v+<w<u-} we have that
\begin{equation}\label{bc2}
w_t > v \quad \text{on} \ B_{9/10} \times \{L-1/2\} \quad \text{for} \ t
\leq \frac{c_2}{4}e^{-A(L-1/2)}\eps = c_4\epsilon.
\end{equation}
Let
$E= \{t \in [0,c_4\epsilon] : v \leq w_t \ \text{in} \ \overline{\mathcal{C}_{1}}\}$.
We claim that $E=[0,c_4\epsilon].$ Indeed, $0 \in E$ and clearly $E$
is closed. We need to show that $E$ is open. Let $t_0 \in E$, then
since $w_{t}$ is superharmonic in its positive phase  and
satisfies \eqref{bc1}-\eqref{bc2} we only need to show that
$w_{t_0}>v=0$ on $F(v)\cap \mathcal{C}_1$.

In the case $t_0=0$, $w_0 = w > 0$ on $F(v)$ follows from the assumption
that $F(u)$ is a graph in the vertical direction.  In fact for all $t$,
$w_t= w >0$ on $F(v) \cap (\mathcal C_1 \backslash \mathcal C_3)$ because
$g$ is zero there.  It remains to rule out the case, in which $t_0>0$,
and $F(v)$ touches $F(w_{t_0})$ in $\mathcal C_3$, that is, where
$g(x_0)\neq0$.

Suppose by contradiction $x_0 \in F(v) \cap F(w_{t_0})\cap \mathcal C_3$,
$t_0>0$.  If $\nabla w_{t_0}(x_0) \neq 0$, then by the implicit function
theorem $F(w_{t_0})$ is smooth in a neighborhood of $x_0$ and
hence there exists an exterior tangent ball $\mathcal{B}$ at $x_0$ for
$F(v)$. Therefore, for $\nu$ the outward unit normal to $\mathcal{B}$ at
$x_0$ we have that
$$w_{t_0}(x) \geq v(x) = (x-x_0,\nu)^+ + o(|x-x_0|)$$
as $x \rightarrow x_0$, contradicting \eqref{smallerthan1}.

On the other hand, if $\nabla w_{t_0}(x_0)=0$, then in a small neighborhood
$\mathcal{B}_r(x_0)$ we have that $$w_{t_0}(x) \leq Cr^2.$$ However, according to the corkscrew condition,
there exists a ball
$\mathcal{B}_{\delta r}(y) \subset \mathcal{B}_r(x_0)\cap \mathcal{C}_1^+(v)$, for some small $\delta>0$.
By the non-degeneracy of $v$ we then obtain
$$\sup_{\mathcal{B}_r(x_0)} v \geq c r,$$  and
again we reach a contradiction.

Thus $c_4\eps \in E$, and
\[
v \leq w_{c_4 \eps} \ \ \text{on}  \ \ \overline{\mathcal{C}}_1.
\]
Hence, according to the definition of $g$,
\[
\{w \leq c_4 e^{-AL}\eps \}\cap \{|x'|< 1/2\} \subset \{v=0\} \cap \{|x'|< 1/2\}.
\]
Moreover, by \eqref{wislip}
$$w \leq C_1 d(x).$$
Thus,
$$\{d(x)\leq c_6 \eps \}\cap \{|x'|<1/2\} \subset \{v=0\} \cap \{|x'|<1/2\}.$$This
implies the Lipschitz continuity of $F(u) \cap \{|x'|< 1/2\}$ with
bound depending only on $L,n$ and the NTA constants.
\qed

\section{A priori gradient bound for minimal surfaces}

In this section we present our proof of Theorem \ref{boundMS}.
Recall that  $B_r$ denotes an open $(n-1)$-dimensional ball of
radius $r$, while $\mathcal{B}_r$, denotes an open $n$-dimensional
ball of radius $r$.

Our proof is parallel to one in the free boundary
setting above.  One main ingredient which will
allow us to apply our deformation argument will be the (weak)
Harnack inequality for solutions to elliptic equations on
minimal surfaces due to Bombieri and Giusti \cite{BG}. We recall
its statement in the form in which we will use it later in the proof.


Let $\Delta_S$ denotes the Laplace-Beltrami operator on the
surface $S$.


\begin{thm}\label{weakHarnack}
Let $p< \dfrac{n-1}{n-3}$. There is a constant $C(p)<\infty$ and $\beta>0$
depending on dimension such that if $S$ is an area minimizing hypersurface
in $\mathcal B_R = \mathcal B_R(x_0)$ and $x_0\in S$ and
$v$ is a positive supersolution to the Laplace-Beltrami operator,
$\Delta_S v \le 0$, in $\mathcal B_R \cap S$, then
\begin{equation}
\left(\fint_{\mathcal B_r \cap S}v^p dH_{n-1}\right)^{1/p} \leq C(p)
\inf_{\mathcal B_r \cap S} v
\end{equation}
for all $r \leq \beta R$.
\end{thm}

\begin{cor}\label{global Harnack}Let $S$ be an oriented surface of least
area in $B_1 \times \mathbb{R}\subset \mathbb{R}^{n-1}\times \mathbb{R}$.
Assume $\mathcal{S}_{1/2}:=S
\cap
(B_{1/2} \times \mathbb{R})$ is connected, and $S \subset
B_1
\times
[-M,M]$. Let $v$ be a positive supersolution to to the Laplace-Beltrami operator,
$\Delta_S v \le 0$, in $(B_1 \times \mathbb{R}) \cap S$,  such
that \begin{equation}\label{average}\int_{\mathcal{S}_{1/2}}v dH_{n-1}
\geq 1.
\end{equation} Then \begin{equation} v \geq c \quad \textrm{on} \quad
\mathcal{S}_{1/2}, \end{equation} with $c>0$ depending only on $n$ and $M.$
\end{cor}

\begin{proof}  Let $\beta$ (small) be
the constant in Theorem \ref{weakHarnack}. Decompose
$\R^{n}$ into cubes of side-length $\beta/(20\sqrt{n}).$ For
each cube $Q_i$ that intersects $\mathcal{S}_{1/2}$ take a ball
$\widetilde{B}_i \supset Q_i$ with center $x_i$ on
$\mathcal{S}_{1/2} \cap Q_i$ and radius $\beta/20$. Clearly, the
number $N$ of balls $\widetilde{\mathcal{B}}_i$ that cover $\mathcal{S}_{1/2}$
depends only on $n$ and $M.$

We say that  $\widetilde{\mathcal{B}}_i \sim
\widetilde{\mathcal{B}}_j$ if there exists a chain of balls
$\widetilde{\mathcal{B}}_k$ connecting $\widetilde{\mathcal{B}}_i$
and $\widetilde{\mathcal{B}}_j$ such that consecutive balls
intersect. This defines an equivalence relation. To each
equivalence class we can associate the open set which is the union
of all the elements in the class. Notice that open sets
corresponding to distinct equivalence classes are disjoint. Since
$\mathcal{S}_{1/2}$ is connected, we conclude that all the balls
belong to the same equivalence class.

If $\widetilde{\mathcal{B}}_1$ and $\widetilde{\mathcal{B}}_2$
intersect then they are both contained in
$B_{\beta/2}(x_1)$. Hence applying Theorem \ref{weakHarnack} we obtain
\begin{equation}\label{comparable}
 \int_{\widetilde{\mathcal{B}}_1 \cap S} v dH_{n-1} \leq C_0 \fint_{B_{\beta/2} \cap
S} v dH_{n-1} \leq C_1 \inf_{S_{\beta/2}(x_1)} v \leq C_2
\int_{\widetilde{\mathcal{B}}_2 \cap S} v dH_{n-1}.
\end{equation}
In the last inequality we used the well-known fact that
\begin{equation}\label{fatness}
H_{n-1}(S\cap B_\rho(x))
\approx \rho^{n-1} \quad \textrm{for all} \quad  x \in
S.
\end{equation}
It also follows from \eqref{average} that at least one of the balls,
say $\widetilde{\mathcal{B}}_1$, satisfies
\begin{equation}\label{average2}
\int_{\mathcal{S}_{1/2} \cap
\widetilde{\mathcal{B}}_1}v dH_{n-1} \geq 1/N.
\end{equation}
Combining
\eqref{comparable}-\eqref{average2} with the fact that any two balls can
be connected by a chain of length at most $N$, we obtain the desired
conclusion.
\end{proof}

\textit{Proof of Theorem $\ref{boundMS}$}. In what follows, the
constants $c, c_i, C, C_i$ depend only on $n$ and $M$. Denote by $S$ the graph of $\phi$ over $B_1.$ We
present the proof in three steps.

\medskip

\noindent \textbf{Step 1: Separation on a set of substantial measure.}

\medskip

Let $\epsilon >0$ and set
\[
S_\epsilon:=\{(x,\phi(x)+\epsilon) : x
\in B_1 \}.
\]
\noindent We will prove that there exists a smoothly bounded,
closed set $\tilde{E} \subset
B_{1/2}$ of positive measure independent of $\eps$ as $\eps \to 0$
such that
\begin{equation}\label{E}
\text{dist}((x,\phi(x)), S_\epsilon) \geq c_0\epsilon
\quad \text{for all} \ x\in \tilde{E} + B_\delta
\end{equation}
where $\delta>0$ depends on the (a priori) bound on the modulus of continuity 
of $\nabla \phi$. 

Let $\eta \in C_0^\infty(B_1)$ be a smooth cut-off function
such that $\eta \equiv 1$ on $B_{1/2}$. Then, since $\phi$
satisfies \eqref{MS} we have that

$$\int_{B_1}\frac{\nabla \phi \cdot \nabla
(\eta^2 \phi) }{\sqrt{1+|\nabla \phi|^2}} dx =0.$$

\noindent Hence,
\begin{align*}\int_{B_1}\eta^2 \frac{|\nabla
\phi|^2}{\sqrt{1+|\nabla
\phi|^2}}dx =& - 2\int_{B_1}\phi \eta \frac{\nabla \phi \cdot \nabla
\eta}{\sqrt{1+|\nabla \phi|^2}} \leq \\ & \
2\left(\int_{B_1}\frac{\eta^2 |\nabla \phi|^2}{\sqrt{1+|\nabla
\phi|^2}}\right)^{1/2}\left(\int_{B_1}\frac{\phi^2 |\nabla
\eta|^2}{\sqrt{1+|\nabla \phi|^2}}\right)^{1/2}. \end{align*}

\noindent Thus,
\begin{equation*}\int_{B_1}\eta^2 \frac{|\nabla
\phi|^2}{\sqrt{1+|\nabla
\phi|^2}} dx
 \leq 4 \int_{B_1} \phi^2 \frac{|\nabla \eta|^2}{\sqrt{1+|\nabla
\phi|^2}}
 dx\leq C M^2
\end{equation*}

\smallskip

\noindent Since $\eta \equiv 1$ on $B_{1/2}$ we then get
\begin{equation}
\int_{B_{1/2}} |\nabla \phi| dx \leq C_0.
\end{equation}
with $C_0$ depending on $M$ and $n$ only.
Hence, by Chebyshev's inequality, (for $C_1 = 2C_0/|B_{1/2}|$)
$$|\{x \in B_{1/2} : |\nabla \phi| < C_1\}|
 \geq |B_{1/2}|/2.$$
Since $\phi$ is smooth, there is a closed, smoothly bounded set 
\[
\tilde E \supset \{x \in B_{1/2} : |\nabla \phi| < C_1\}
\]
and $\delta >0$ sufficiently small depending on the modulus of continuity
of $\nabla \phi$ such that 
\[
\tilde E + B_\delta \subset\{x \in B_{1/2} : |\nabla \phi|^2 \leq C_1^2 + 1\}
\]
This implies the desired claim \eqref{E}, for
small enough $\epsilon$ and $\delta$, 
depending on the smoothness of $\phi$.

In what follows we denote by $E = \{(x,\phi(x)), x\in \tilde{E}\}.$
Clearly, $H_{n-1}(E) \geq |B_{1/2}|/2.$

\medskip

\noindent \textbf{Step 2: Construction of a family of
subsolutions}.

\medskip

For the time being let $S$ be any smooth surface.  Denote by $H(P, S)$ the mean curvature of  $S$ at a point $P \in S,$ (i.e. the trace of the second fundamental form of $S$ at $P$.) Assume that $S$ is a smooth graph over $B_1$, i.e.
$S= \{(x, \phi(x)) : x \in B_1\}$, and let
$w$ be a $C^2$ non-negative function on $S$. Consider
the surface $S_{t,\nu}:=S+tw\nu$ obtained deforming $S$
along the upward unit normal to $S$, 
that is $$S_{t,\nu} = \{(x,\phi(x))+ tw(x,\phi(x))\nu_x, x \in B_1\},$$ with $$\nu_x = \dfrac{(-\nabla \phi(x), 1)}{\sqrt{1+|\nabla \phi(x)
|^2}}.$$ Then, for $t$ small enough, $S_{t}$ is also a
graph
 and one can compute
(see for example \cite{Ko})
\begin{align}\label{korevar}
&H(P_t, S_{t,\nu}) = H(P, S) + t(\Delta_{S} w(P) +
|A|_{S}^2 w(P)) + O(t^2), \\ &P:= (x,\phi(x)), P_t:=(x,\phi(x))+ tw(x,\phi(x))\nu_x, 
\end{align}

\smallskip

\noindent where $|A|_S$ is the norm of the second
fundamental form of $S$. (The $O(t^2)$ term depends
at most on the third
derivatives of $\phi$ and on the second derivatives of $w$.)

Applying formula \eqref{korevar} to our minimal surface $S$ we
find that
\begin{equation}\label{korevar2}
H(P_t, S_{t,\nu}) = t(\Delta_{S} w(P) + |A|_{S}^2 w(P)) + O(t^2).
\end{equation}

In order to run a continuity argument (as in the proof of Theorem
\ref{main}), we wish to use formula \eqref{korevar2} to produce a
family of surfaces $S_{t,\nu} $ which are strict subsolutions to the
minimal surface equation i.e. $H(\cdot, S_{t,\nu})
>0$ at least outside $E_{t,\nu}:= E+ tw\nu$, with $E$ the set from the previous step. Towards
this aim we prove the following claim.

\medskip

\noindent \textit{Claim.} There exists a function $w$ defined on
$S$  such that
\begin{align*}
& \Delta_S w + |A|_S^2 w > 0 \quad \text{on} \ S \setminus E,\\
& w(x,\phi(x)) =1 \quad \text{on} \ \tilde{E},\\
& w(x,\phi(x))=0 \quad \text{on} \ \partial B_1.
\end{align*}

\noindent Moreover
\begin{equation}\label{lowerbound}
w(x,\phi(x)) \geq c_0 >0 \quad \text{on} \ B_{1/2},
\end{equation}
with $c_0$ depending only on $n, M$ and $w \in C^2(\overline{S\setminus E})$ with $C^2$ bounds depending on $S$ and $E$.

\noindent \textit{Proof of the claim.} Let $w_1$ be the solution
to the following boundary value problem,
\begin{align*}
& \Delta_S w_1 = 0 \quad \text{on} \ S \setminus E,\\
& w_1(x,\phi(x)) =1 \quad \text{on} \ \p \tilde{E},\\
& w_1(x,\phi(x)) =0 \quad \text{on} \ \partial B_1.
\end{align*}
Note that the solution exists and is smooth in its domain of definition
because $\tilde E$ is smoothly bounded.  
Extend $w_1  =1$ on $\tilde E.$ Then $\Delta_S w_1 \leq 0$ on $S$. 
Moreover, according to Step 1,
we have that (using the notation of Corollary \ref{global Harnack})
\[
\int_{\mathcal{S}_{1/2}} w_1 dH_{n-1} \geq \int_{E} w_1 dH_{n-1} 
=  H_{n-1}(E) \geq |B_{1/2}|/2.
\]
Hence we can apply
Corollary \ref{global Harnack} to conclude that
\begin{equation}\label{w1c} w_1 \geq c \quad \text{on} \quad \mathcal{S}_{1/2}=S \cap (B_{1/2} \times
\mathbb{R}).\end{equation}

Now, let $w_0$ be the solution to the following problem:
\begin{align*}
& \Delta_S w_0 = 1 \quad \text{on} \ S \setminus E,\\
& w_0(x,\phi(x)) =0 \quad \text{on} \ \tilde{E},\\
& w_0(x,\phi(x))=0 \quad \text{on} \ \partial B_1.
\end{align*}
and set $w=w_1 + \delta_1 w_0$. Clearly, $|\nabla w_0|$ is bounded 
(by a constant depending on $S$ and $E$). Applying Hopf's lemma 
to $w_1$ on $(\p B_1 \times \mathbb{R}) \cap S$ , we obtain that, 
for $\delta_1$ sufficiently small, $w > 0$ in a neighborhood of 
$(\p B_1 \times \mathbb{R}) \cap S$  and hence (for a possibly smaller $\delta_1$) 
$w>0$ on $S$. Moreover, in view of $\eqref{w1c}$, we can choose
$\delta_1$ so that $w$ satisfies \eqref{lowerbound}. Thus, $w$
has all the required properties.

\smallskip

In view of the claim, according to formula \eqref{korevar2}, if $t$ is
sufficiently small, $0<t \leq \epsilon_0$ then $$H(\cdot, S_{t,\nu}) >0, \quad \text{on $S_{t,\nu} \setminus E_{t,\nu}$}.$$

\medskip

\noindent \textbf{Step 3: Comparison.}

\medskip

We show that for $0 \leq t \leq
c_0\epsilon \leq \epsilon_0$, the surface $S_{t,\nu}$ is below the surface
$S_\epsilon.$ Indeed, this is true at $t=0$. The first touching point
cannot occur at some $x \in \partial B_1$, as our deformation
leaves the $\partial B_1$ fixed. Moreover, for $t$ small enough,
no touching can occur on 
$E_{t,\nu}$ in view of \eqref{E} in Step 1.  
Finally $S_{t,\nu}$ is a 
strict subsolution on $S_{t,\nu} \setminus E_{t,\nu}$, hence no touching can occur
there either. Since $w$ satisfies \eqref{lowerbound}, we can then conclude
that for all sufficiently small $\epsilon,$ 
(recall $S_{c_0\epsilon,\nu}=S+c_0\epsilon w\nu $) 
$$\text{dist}((x,\phi(x)), S_\epsilon) \geq
\text{dist}((x,\phi(x)), S_{c_0\epsilon,\nu}) \geq c_1 \epsilon \quad \text{on}
\quad B_{1/2}$$ as desired.  Note that although the size of $\epsilon_0$ 
depends on the a priori bound on $\nabla \phi$, the constants $c_0>0$ and  
$c_1>0$ do not. 

 \qed

\section{Final Remarks}

The analogy between the two gradient bound proofs presented here goes farther.
Not only does each proof depend crucially on  a scale-invariant Harnack inequality for
the second variation operator of the associated functional, but also
the proofs of these two Harnack estimates follow a roughly parallel course. 

The key ingredient of our proof of the gradient bound for minimal surface graphs
is the Harnack inequality for the Laplace-Beltrami operator on the surface.
This Harnack inequality permits us to convert a gradient bound on average (separation on
a set of substantial measure) to a gradient bound everywhere (separation everywhere).
The way this Harnack inequality is proved by Bombieri and Giusti is as follows.
A monotonicity formula yields (via a limiting cone argument) a  measure-theoretic
form of connectivity.  This, in turn, implies another 
scale-invariant form of connectivity, an isoperimetric, or Poincar\'e-type, inequality.  One then
deduces a Harnack inequality for the Laplace-Beltrami operator on the minimal surface
by a Moser-type argument.  

In the free boundary case, a monotonicity formula due to
Alt, Caffarelli and Friedman yields (by arguments of \cite{ACS} and \cite{D1}) the NTA property,
a scale-invariant form of connectivity.   A theorem of \cite{JK} says that the 
NTA property implies a boundary Harnack inequality.   The boundary Harnack
inequality is used to show that separation of level surfaces of the solution function
$u$ at distances far from the free boundary implies  a similar separation all the way up
to the free boundary.  

The parallel between these two Harnack inequalities leads to the hope
that there is  a Harnack estimate for the second variation operator associated 
to minimizers of functionals of the form
\[
\int  |\nabla v|^2 + F(v)
\]
for wider classes of functions $F$.

\end{document}